\documentclass[10pt,reqno]{amsart}
\usepackage{amssymb,amsmath,hyperref,verbatim,cite,graphicx}
\makeatletter \let\cl@chapter\relax \makeatother
\usepackage{cleveref}
\usepackage{multirow}
\usepackage{longtable}
\usepackage{geometry}
\usepackage{pdflscape}
\usepackage{xcolor}
\usepackage[colorinlistoftodos,prependcaption,textsize=small]{todonotes}
\usepackage{bcol}

\newcommand*{\I}{\imath}%
\crefname{cons}{constraint}{constraints}
\crefname{fig}{figure}{figures}
\crefname{claim}{claim}{claims}

\newtheorem{thm}{Theorem}
\newtheorem{lemma}[thm]{Lemma}
\newtheorem{proposition}[thm]{Proposition}
\newtheorem{claim}[thm]{Claim}
\newtheorem{corollary}[thm]{Corollary}

\theoremstyle{definition}

\newtheorem{remark}{Remark}

\begin{document}

\title[SBC for Nonconvex QCQP with Bounded Complex Variables]{A Spatial Branch-and-Cut Method for Nonconvex QCQP with Bounded Complex Variables}

\author{Chen Chen \and Alper Atamt\"urk \and Shmuel S. Oren}

\affiliation{C. Chen}{Industrial Engineering and Operations Research, University of California, Berkeley, CA 94720-1777.\,\,\,email:\,\texttt{{chenchen@berkeley.edu}}}

\affiliation{A. Atamt\"{u}rk}{Industrial Engineering and Operations Research, University of California, Berkeley, CA 94720-1777.\,\,\,email:\,\texttt{{atamturk@berkeley.edu}}}

\affiliation{S. S. Oren}{Industrial Engineering and Operations Research, University of California, Berkeley, CA 94720-1777.\,\,\,email:\,\texttt{{oren@berkeley.edu}}}

\maketitle

\begin{abstract}
We develop a spatial branch-and-cut approach for nonconvex Quadratically Constrained Quadratic Programs with bounded complex variables (CQCQP). Linear valid inequalities are added at each node of the search tree to strengthen semidefinite programming relaxations of CQCQP. These valid inequalities are derived from the convex hull description of a nonconvex set of $2 \times 2$ positive semidefinite Hermitian matrices subject to a rank-one constraint.  We propose branching rules based on an alternative to the rank-one constraint that allows for local measurement of constraint violation.  Closed-form bound tightening procedures are used to reduce the domain of the problem.  We apply the algorithm to solve the Alternating Current Optimal Power Flow problem with complex variables as well as the Box-constrained Quadratic Programming problem with real variables.

\end{abstract}

\begin{center} August 2015; July 2016 \end{center}

\BCOLReport{15.04}

\section{Introduction}
\label{sec:intro}
The nonconvex quadratically-constrained quadratic programming problem with complex bounded variables (CQCQP) has numerous applications in signal processing \cite{waldspurger2015phase,de2011design,huang2014randomized} and control theory \cite{ben2003extended}, among others. Our main motivation for developing an algorithm for CQCQP is to solve power flow problems with alternating current \cite{lavaei_zero_2012,jabr2008opf}. We consider the following formulation of CQCQP: 

\begin{alignat*}{2}
\min   x^*Q_0x + \mbox{Re}(c_0^*x) + b_0 \ \ \ \ \ \  & \\
(\mathbf{CQCQP}) \ \ \ \text{s.t.} \ \   x^*Q_ix + \mbox{Re}(c_i^*x) + b_i\leq 0,  & \ i = 1,...,m\\
\ell \leq x \leq u&\\
x\in\mathbb{C}^n. \ \ & \\
\end{alignat*}
We denote the conjugate transpose operator by $^*$, and real components with $\mbox{Re($\cdot$)}$ and imaginary components with $\mbox{Im($\cdot$)}$. The decision vector $x\in \mathbb{C}^n$ has complex entries, and the remaining terms are data: Hermitian matrices $Q_i \in \mathbb{H}^{n\times n}$, real vector $b \in \mathbb{R}^n$, and complex vector $c_i \in \mathbb{C}^{n}$. Variable bounds $\ell \leq x \leq u$ are component-wise inequalities in the complex space --- we assume these to be finite. Note that the Hermitian assumption on $Q_i$ is without loss of generality since, otherwise, $Q_i$ may be replaced with $(Q_i+Q_i^*)/2$.



Quadratically-constrained quadratic programs on real variables (RQCQP) can be considered a special case of CQCQP where all imaginary components are restricted to 0. Likewise, CQCQP can be solved using RQCQP by defining separate real decision vectors to represent $\mbox{Re}(x)$ and $\mbox{Im}(x)$.  Despite this modeling equivalence, it is beneficial to exploit the structure of CQCQP to derive complex valid inequalities, bound tightening and spatial partitioning rules.
Stronger results can be obtained in the complex space compared to their real counterparts. In particular,
in \Cref{subsec:VI} we show that the new valid inequalities derived in the complex space can be interpreted as a complex analogue of the RLT inequalities (e.g. \cite{anstreicher2009semidefinite}), and that the RLT inequalities applied to the RQCQP reformulation of CQCQP are dominated by these complex valid inequalities.

Observing the advantages of working with the complex formulations, a growing body of work considers CQCQP as a problem distinct from RQCQP. For instance, Josz and Molzahn \cite{josz2015moment} show that moment-based semidefinite programming relaxations derived from CQCQP are not, in general, equivalent to those derived from a RQCQP transformation of the same problem and with computational experiments they demonstrate the power of retaining the CQCQP formulation. Stronger results are given regarding the SDP relaxation of the CQCQP associated with the S-lemma problem instead of the RQCQP equivalent \cite{huang2007complex,beck2006strong}.  Jiang et al. \cite{jiang2014approximation} demonstrate theoretically and empirically the benefits of working with CQCQP representation for generating strong relaxations. Kocuk et al. \cite{kocuk2014inexactness} demonstrate both theoretically and empirically for the optimal power flow problem that the spatial branch-and-bound approach can be improved by considering the complex formulation before converting it to a problem in the reals.


In this paper we give a spatial branch-and-cut (SBC) approach to solve CQCQP.  For brevity, we assume familiarity with the general spatial branching framework; the reader is referred to Belotti et al. \cite{belotti2013mixed} for a thorough treatment on the subject.  There are several spatial branching algorithms for RQCQP (e.g. \cite{linderoth_simplicial_2005,misener2013glomiqo,raber1998simplicial,bao2009multiterm,phan_lagrangian_2010}); however, to the best of our knowledge this paper present the first spatial branching algorithm developed specifically for CQCQP.  Our SBC algorithm has three distinguishing features. First, we use a complex SDP formulation strengthened with valid inequalities. These inequalities are derived from the convex hull description of a nonconvex set of $2 \times 2$ positive semidefinite Hermitian matrices subject to a rank-one constraint that arises from a lifted formulation of CQCQP. Their relationship with the RLT inequalities are discussed at the end of \Cref{subsec:VI}. Second, we propose spatial partitioning on the entries of the relaxation's decision matrix, and we consider branching rules based on an alternative measure of constraint violation in lieu of the matrix rank constraint. This method does not rely on the CQCQP structure, and applies generally to any rank-constrained SDP formulation. Therefore, one can configure our algorithm to solve RQCQP using an SDP relaxation strengthened with RLT inequalities. Third, we develop bound tightening procedures based on closed-form solutions.

Computational experiments are conducted on the Alternating Current Optimal Power Flow (ACOPF) problem and the Box-constrained Quadratic Programming (BoxQP) problem. ACOPF is a generation dispatch problem that models alternating current (AC) using steady-state power flow equations, which are nonconvex quadratic constraints on complex variables relating power and voltage at buses and across transmission lines. ACOPF is commonly solved with heuristic iterative Newton-type solvers (e.g.  \cite{sun1984optimal}).  ACOPF may be modeled as a CQCQP problem, and there has been a recent interest in solving this formulation to optimality with SDP relaxations (e.g. \cite{bai_semidefinite_2008,lavaei_zero_2012}) and branch-and-bound methods (e.g. \cite{phan_lagrangian_2010,gopalakrishnan2012global,kocuk2014inexactness}) due to the potential for establishing global optimality.
We show with numerical experiments that the formulation exploiting the complex structure leads to significantly faster solution times compared to the application of the algorithm on RQCQP reformulations on ACOPF instances. Since the proposed branching rules also apply to RQCQP, we test them on BoxQP -- a well-studied nonconvex quadratic programming problem with nonhomogeneous quadratic objective and bounded real variables. BoxQP can be solved with a problem-specific finite branch-and-bound method \cite{burer2009globally}; therefore, it provides a useful benchmark for the general approach presented here.  We demonstrate that the proposed method can also be used as a viable RLT and SDP-based RQCQP solver.

The rest of the paper is organized as follows: \Cref{sec:spbb} details the spatial branch-and-cut algorithm with three major components: valid inequalities from the convex hull description of rank-one constrained relaxations, branching on matrix entries, and bound tightening procedures; \Cref{sec:experiments} contains the results from computational experiments with ACOPF and BoxQP problems; \Cref{sec:conclusion} concludes the paper.


\section{The Spatial Branch-and-Cut Method}
\label{sec:spbb}
The convex relaxation of CQCQP considered comes from the rank-one constraint of a standard SDP reformulation  \cite{lovasz1991cones,shor_quadratic_1987}, often called Shor's relaxation:
\begin{subequations}
\begin{alignat}{2}
 \min \ & \langle Q_0X \rangle + \mbox{Re}(c_0^*x) + b_0 \\
\bold{(CSDP}) \ \ \ \text{s.t.} \ &\langle Q_i,X \rangle + \mbox{Re}(c_i^*x) + b_i \leq 0, \quad i = 1,...,m \label{eq:liftlin}\\
&\ell \leq x \leq u \label{eq:vbounds}\\
&\left[\begin{array}{cc}
1 & x^*\\
x & X\end{array}\right]\succeq 0, \label{eq:PSD}
\end{alignat}
\end{subequations}
where $X \in \mathbb{H}^{n\times n}$ is a Hermitian submatrix of decision variables.
Imposing a rank-one constraint on the matrix
$Y:=\left[\begin{array}{cc}
1 & x^*\\
x & X\end{array}\right]$
gives an equivalent reformulation of CQCQP.  Let $y:= \left(\begin{array}{cc}
1 \\
x \end{array}\right)$ so that $Y=yy^*$. Furthermore, denote the real and imaginary components as $Y := W+\I T$ and $y:=w+\I t$. Since $Y$ is Hermitian we have $W_{ij}=W_{ji}, \ T_{ij}=-T_{ji}, \ \forall i,j$ and, thus, $\mbox{diag}(T)=0$.

This section is divided into three subsections.  First, we derive valid inequalities to strengthen the SDP relaxation.  Second, we propose a methodology for branching on the entries of matrix $Y$.  Third, we present closed-form bound-tightening techniques.

\subsection{Valid Inequalities}

\label{subsec:VI}
 Here we describe the convex hull of a nonconvex set that is applicable to any $(i,j)$ entry of $Y$, where $1 \leq i \neq j\leq n+1$. Let $\mathcal{J}_C$ be the set of feasible solutions in $\mathbb{C}^{2\times 2}$ to the following constraints:
\begin{subequations}
\begin{alignat}{2}
L_{ii} & \leq W_{ii} \leq U_{ii}, \label{eq:RSOC3a}\\
L_{jj} & \leq W_{jj} \leq U_{jj}, \label{eq:RSOC3b}\\
L_{ij} W_{ij} & \leq \, T_{ij} \, \leq  U_{ij}W_{ij}, \label{eq:RSOC4}\\
&W_{ij}\geq 0, \label{eq:RSOCnn}\\
W_{ii} W_{jj} &  =W_{ij}^2 + T_{ij}^2, \label{eq:RSOC1}
\end{alignat}
\end{subequations}
where $L$ and $U$ are real, symmetric bound matrices with $L \leq U$ and $L_{ii},L_{jj} \geq 0$.

One can always find such bound values so that $\mathcal{J}_C$ is valid for either CQCQP or an affine transformation of CQCQP. We will give a generic method to obtain such bounds at the end of this subsection.  First we describe the convex hull of $\mathcal{J}_C$.  The convex hull description will provide linear valid inequalities for $\mathcal{J}_C$, which can be applied to all $2\times 2$ principal minors of $Y$ in order to strengthen CSDP.

In order to describe the convex hull of $\mathcal{J}_C$ it is convenient to use the following sigmoid function:
\[f(x):= \left\{ \begin{array}{cc}
(\sqrt{1+x^2}-1)/x, & x \neq 0 \\
0, & x=0\\ \end{array}.\right.\]

\begin{remark}
\label{rem:fn}
$f(x)$ is increasing, strictly bounded above by $+1$ and strictly bounded below by $-1$.
\end{remark}

Consider the following linear inequalities:
\begin{subequations}
\begin{alignat}{2}
&\pi_0 + \pi_1 W_{ii} + \pi_2 W_{jj} + \pi_3 W_{ij} + \pi_4 T_{ij} \geq U_{jj}W_{ii}+U_{ii}W_{jj}-U_{ii}U_{jj}, \label{eq:vi1}\\
&\pi_0 + \pi_1 W_{ii} + \pi_2 W_{jj} + \pi_3 W_{ij} + \pi_4 T_{ij} \geq L_{jj}W_{ii}+L_{ii}W_{jj}-L_{ii}L_{jj}, \label{eq:vi2}
\end{alignat}
\end{subequations}

\noindent
where the coefficients $\pi$ are defined as
\begin{alignat*}{2}
&\pi_0 := -\sqrt{L_{ii}L_{jj}U_{ii}U_{jj}},\\
&\pi_1 := -\sqrt{L_{jj}U_{jj}},\\
&\pi_2 := -\sqrt{L_{ii}U_{ii}},\\
&\pi_3 := (\sqrt{L_{ii}}+\sqrt{U_{ii}})(\sqrt{L_{jj}}+\sqrt{U_{jj}})\frac{1-f(L_{ij})f(U_{ij})}{1+f(L_{ij})f(U_{ij})},\\
&\pi_4 := (\sqrt{L_{ii}}+\sqrt{U_{ii}})(\sqrt{L_{jj}}+\sqrt{U_{jj}})\frac{f(L_{ij})+f(U_{ij})}{1+f(L_{ij})f(U_{ij})}.
\end{alignat*}

\begin{lemma}
\label{lem:angles}
For  $\alpha_{ij} \in \{L_{ij},U_{ij}\}$ we have
\begin{equation}
1-f(L_{ij})f(U_{ij})+\alpha_{ij}(f(L_{ij})+f(U_{ij}))=(1+f(L_{ij})f(U_{ij}))\sqrt{1+\alpha_{ij}^2}.
\label{eq:vilemma}
\end{equation}
\end{lemma}
\begin{proof}
If $\alpha_{ij}=0$, then equality~(\ref{eq:vilemma}) follows immediately.  Otherwise, suppose $\alpha_{ij}=L_{ij}\neq 0$.  Then we have

\begin{alignat*}{2}
&1-f(L_{ij})f(U_{ij})+L_{ij}(f(L_{ij})+f(U_{ij}))-(1+f(L_{ij})f(U_{ij}))\sqrt{1+L_{ij}^2}\\
=\ &-f(L_{ij})f(U_{ij})+L_{ij}f(U_{ij})-f(L_{ij})f(U_{ij})\sqrt{1+L_{ij}^2}\\
=\ &f(U_{ij})[L_{ij}-f(L_{ij})(1+\sqrt{1+L_{ij}^2})]\\
=\ &f(U_{ij})[L_{ij}-L_{ij}^2/L_{ij}]\\
=\ &0.
\end{alignat*}

\noindent
Therefore
\[1-f(L_{ij})f(U_{ij})+L_{ij}(f(L_{ij})+f(U_{ij}))=(1+f(L_{ij})f(U_{ij}))\sqrt{1+L_{ij}^2}.\]

\noindent
The case where $\alpha_{ij}=U_{ij}\neq 0$ follows by symmetry.
\end{proof}

\begin{proposition}
\label{lem:valid}
Inequalities~(\ref{eq:vi1}) and (\ref{eq:vi2}) are valid for $\mathcal{J}_C$.

\end{proposition}
\begin{proof}
\begin{subequations}
For any convex function $f$, the secant line connecting $(a,f(a))$ and $(b,f(b))$ lies above the graph of $f$.  Thus for any variable $q$ we have $\ell_q \leq q \leq u_q \implies (\ell_q+u_q)q - \ell_q u_q \geq q^2$.
Now for $k\in \{i,j\}$, since  $\sqrt{L_{kk}}\leq \sqrt{W_{kk}} \leq \sqrt{U_{kk}}$, applying this secant principle yields

\begin{equation}
(\sqrt{L_{kk}}+\sqrt{U_{kk}})\sqrt{W_{kk}}\geq\sqrt{L_{kk}U_{kk}}+W_{kk}.
\label{eq:secmat}
\end{equation}

\noindent
Multiplying inequalities from~(\ref{eq:secmat}) for $k=1$ and $k=2$ gives

\[\sqrt{W_{ii}W_{jj}}(\sqrt{L_{ii}}+\sqrt{U_{ii}})(\sqrt{L_{jj}}+\sqrt{U_{jj}})\geq (\sqrt{L_{ii}U_{ii}}+W_{ii})(\sqrt{L_{jj}U_{jj}}+W_{jj}).\]
Rearranging terms, we have

\begin{alignat}{2}
&\sqrt{W_{ii}W_{jj}}(\sqrt{L_{ii}}+\sqrt{U_{ii}})(\sqrt{L_{jj}}+\sqrt{U_{jj}})+\pi_0+\pi_1W_{ii}+\pi_2 W_{jj} \geq W_{ii}W_{jj}. \label{eq:ncvxvi}
\end{alignat}
The right-hand-sides of constraints (\ref{eq:vi1}) and (\ref{eq:vi2}) are RLT inequalities (discussed in the next subsection) and thus are valid for the bilinear term $W_{ii}W_{jj}$, which is the right-hand-side of the valid inequality~(\ref{eq:ncvxvi}).  It remains to show that the left-hand-side of~(\ref{eq:ncvxvi}) is overestimated by the left-hand sides of~(\ref{eq:vi1}) and (\ref{eq:vi2}), i.e.:

\begin{alignat}{2}
\pi_3W_{ij}+\pi_4T_{ij} \geq \sqrt{W_{ii}W_{jj}}(\sqrt{L_{ii}}+\sqrt{U_{ii}})(\sqrt{L_{jj}}+\sqrt{U_{jj}}). \label{eq:wtslhs}
\end{alignat}

Let $T_{ij}=\alpha W_{ij}$ for some $L_{ij}\leq\alpha_{ij}\leq U_{ij}$. Note that constraint~(\ref{eq:RSOCnn}) restricts $W_{ij}$ to be nonnegative, so from constraint~(\ref{eq:RSOC1}) we have:

\begin{equation}
\sqrt{W_{ii}W_{jj}} = \sqrt{1+\alpha_{ij}^2}W_{ij}.
\label{eq:alphrank}
\end{equation}
Substituting equality~(\ref{eq:alphrank}) into inequality~(\ref{eq:wtslhs}), we want to show
\begin{equation}
(\pi_3+\alpha_{ij} \pi_4)W_{ij}\geq \sqrt{1+\alpha_{ij}^2}(\sqrt{L_{ii}}+\sqrt{U_{ii}})(\sqrt{L_{jj}}+\sqrt{U_{jj}})W_{ij}
\end{equation}
for any $L_{ij}\leq \alpha_{ij} \leq U_{ij}$.  Replacing the coefficients with their definitions and simplifying, we get the following equivalent condition:

\begin{alignat}{2}
1-f(L_{ij})f(U_{ij})+\alpha_{ij}(f(L_{ij})+f(U_{ij}))-(1+f(L_{ij})f(U_{ij}))\sqrt{1+\alpha_{ij}^2}
\geq 0. \label{eq:wtsfinal}
\end{alignat}

To show that inequality~(\ref{eq:wtsfinal}) is valid, first observe that the second derivative of the left-hand side with respect to $\alpha_{ij}$ is $-(1+f(L_{ij})f(U_{ij}))/(1+\alpha^2)^{3/2}$.  Thus the left-hand side is concave w.r.t $\alpha$ (as noted in \Cref{rem:fn} we have that $|f(L_{ij})|,|f(U_{ij})|<1$).  Therefore, we only need to check that the left-hand side is nonnegative for $\alpha_{ij} \in \{L_{ij},U_{ij}\}$. From \Cref{lem:angles} we have that the left-hand side is exactly zero, and so inequality~(\ref{eq:wtsfinal}) and consequently inequalities~(\ref{eq:vi1}) and (\ref{eq:vi2}) are valid for $\mathcal{J}_C$.
\end{subequations}
\end{proof}

Let $\mathcal{J}_S$ be the set of feasible solutions to the natural SDP relaxation of $\mathcal{J}_C$:
\begin{subequations}
\begin{alignat}{2}
&L_{ii} \leq W_{ii} \leq U_{ii}, \label{eq:SSOC3a}\\
&L_{jj} \leq W_{jj} \leq U_{jj}, \label{eq:SSOC3b}\\
&L_{ij} W_{ij} \leq T_{ij} \leq  U_{ij}W_{ij} \label{eq:SSOC4}\\
&W_{ij}\geq 0, \label{eq:SSOC5}\\
&W_{ii}W_{jj} \geq W_{ij}^2 + T_{ij}^2. \label{eq:SSOC1}
\end{alignat}
\end{subequations}

\begin{remark}
\label{rem:psdsoc}
For $n=2$, $W + \I T \succeq 0$ is equivalent to the principal minor constraints: $W_{ii} \geq 0,W_{jj} \geq 0, W_{ii}W_{jj} \geq W_{ij}^2 + T_{ij}^2$.  Since by construction $L_{ii},L_{jj} \geq 0$, then constraint~(\ref{eq:SSOC1}) is equivalent to the positive semidefinite constraint for $\mathcal{J}_S$.
\end{remark}

Let $\mathcal{J}_V$ be the set of $X$ satisfying inequalities~(\ref{eq:vi1})-(\ref{eq:vi2}).  We shall prove that the convex hull of $\mathcal{J}_C$ can be obtained by adding the valid inequalities~(\ref{eq:vi1})-(\ref{eq:vi2}) to the SDP relaxation.

\begin{proposition}
\textnormal{conv}$(\mathcal{J}_C)= \mathcal{J}_S \cap \mathcal{J}_V$.
\label{prop:convhullj}
\end{proposition}

\begin{proof}
From \Cref{lem:valid} we have that $\mbox{conv}(\mathcal{J}_C)\subseteq \mathcal{J}_S \cap \mathcal{J}_V$.  From constraints~(\ref{eq:RSOC3a})-(\ref{eq:RSOC3b}) and (\ref{eq:RSOC1}) we can observe that $\mathcal{J}_S \cap \mathcal{J}_V$ is bounded.  Thus to prove that $\mathcal{J}_S \cap \mathcal{J}_V \subseteq  \textnormal{conv}(\mathcal{J}_C)$, it is sufficient to ensure that all the extreme points of $\mathcal{J}_S \cap \mathcal{J}_V$ are in $\mathcal{J}_C$.  First, let us invoke the following claim:
\begin{claim}
\label{claim:wclaim}
If $W_{ij}=0$, then either $X \notin \mathcal{J}_S \cap \mathcal{J}_V$ or $X \in \textnormal{conv}(\mathcal{J}_C)$.
\end{claim}

\begin{proof}
By way of contradiction, suppose there exists $\bar{X} \in \mathcal{J}_S \cap \mathcal{J}_V \setminus \mbox{conv}(\mathcal{J}_C)$ such that $\bar W_{ij}=0$.  From constraint~(\ref{eq:RSOC4}) we have that $\bar T_{ij}=0$.  Consequently, if $\bar W_{ii}=0$ or $\bar W_{jj}=0$, then constraint~(\ref{eq:RSOC1}) is satisfied, so either  $\bar{X} \notin \mathcal{J}_S \cap \mathcal{J}_V$ or $\bar X \in \mbox{conv}(\mathcal{J}_C)$.


From \Cref{lem:valid} we only need to consider $\bar W_{ii} \bar W_{jj} > 0$, which
$\bar W_{ii} \bar W_{jj} > 0$ implies that the right-hand side of constraint~(\ref{eq:vi2}) is nonnegative since $0\leq L_{kk} \leq \bar W_{kk}, k\in \{i,j\}$, and so $L_{jj} \bar W_{ii} +L_{ii} \bar W_{jj} - L_{ii}L_{jj} \geq L_{ii} \bar W_{jj} \geq 0$.  On the left-hand side, all terms are nonpositive, so we require all terms to be zero in order to ensure $\bar{X}\in \mathcal{J}_S \cap \mathcal{J}_V$.  Since we have that $\bar W_{ii} \bar W_{jj} > 0$, then $\pi_1\bar W_{ii}=\pi_2 \bar W_{jj}=0$ only if $L_{ii}=L_{jj}=0$.

Now let us define the following two matrices:

\[U^1 := \left[\begin{array}{cc} U_{ii} & 0 \\ 0 & 0\end{array}\right],
U^2 := \left[\begin{array}{cc} 0 & 0 \\ 0 & U_{jj}\end{array}\right].
\]
\noindent
Since $L_{ii}=L_{jj}=0$, then $U^1,U^2 \in \mathcal{J}_C$, and since $\bar W_{ij} = \bar T_{ij} = 0$. Checking constraint~(\ref{eq:vi1}) we have  $0 \geq U_{jj}W_{ii}+U_{ii}W_{jj}-U_{ii}U_{jj}$, so $\bar X \in  \mathcal{J}_S \cap \mathcal{J}_V$ provided $W_{ii} \leq U_{ii},W_{jj}\leq U_{jj}$.  Observe that $\bar X$ can be expressed as the convex combination of $U^1,U^2,$ and the zeros matrix.  $L_{ii}=L_{jj}=0$ implies the zeros matrix is a member of $\mathcal{J}_C$, which in turn implies $\bar X \in \mbox{conv}(\mathcal{J}_C)$, which contradicts our initial assumption.
\end{proof}

Now we will prove that $\mbox{ext}(\mathcal{J}_S \cap \mathcal{J}_V) \in \mathcal{J}_C$.  Observe that if the constraint~(\ref{eq:SSOC1}) is binding at an extreme point of $\mathcal{J}_S \cap \mathcal{J}_V$, then it is a member of $\mathcal{J}_C$.  Moreover, by Claim~\eqref{claim:wclaim}, if a point $\bar X$ with $\bar W_{ij}=0$ is in $\mbox{ext}(\mathcal{J}_S \cap \mathcal{J}_V)$, then $\bar X \in \mbox{conv}(\mathcal{J}_C)$. It follows that $\bar X$ is a member of $\mathcal{J}_C$ since $\mathcal{J}_S \cap \mathcal{J}_V \supseteq \mbox{conv}(\mathcal{J}_C) \supseteq \mathcal{J}_C$.  Therefore, we shall check by cases for any extreme point where constraints~(\ref{eq:SSOC1}) and (\ref{eq:SSOC5}) are not binding.
\\ \\
\noindent \emph{Case 1}: Constraints~(\ref{eq:vi1}) and (\ref{eq:vi2}) are not binding: \\
If constraint~(\ref{eq:SSOC1}) is not binding, then we require at least four linearly independent linear constraints to be binding.  To obtain four such constraints, we require the two variable bounds (\ref{eq:RSOC3a}) and (\ref{eq:RSOC3b}) to be binding.  Moreover, we require that constraint~(\ref{eq:RSOC4}) is binding on both sides with $L_{ij}\neq U_{ij}$, which implies $W_{ij}=T_{ij}=0$. By Claim~\eqref{claim:wclaim}, this point can be disregarded.
\\ \\
\noindent \emph{Case 2}: Constraints~(\ref{eq:vi1}) and (\ref{eq:vi2}) are both binding: \\
Since constraints~(\ref{eq:vi1}) and (\ref{eq:vi2}) share the same coefficients $\pi_3,\pi_4$ for $W_{ij},T_{ij}$, then for an extreme point we require that constraint~(\ref{eq:RSOC4}) is binding on at least one side. Due to Claim~\eqref{claim:wclaim} we need only consider $W_{ij} \neq 0$, so constraint~(\ref{eq:RSOC4}) can count for at most one linearly independent constraint; thus, let $T_{ij}=\alpha_{ij}W_{ij}$, where $\alpha_{ij} \in \{L_{ij},U_{ij}\}$. This gives at most three linearly independent constraints, (\ref{eq:RSOC4}), (\ref{eq:vi1}) and (\ref{eq:vi2}), so at least one of the variable bounds~(\ref{eq:RSOC3a}) and (\ref{eq:RSOC3b}) must be binding.  Define $\alpha_{kk}$ so that $W_{kk}=(U_{kk}-L_{kk})\alpha_{kk}+L_{kk}$ for $k \in \{i,j\}$. Since at least one of $W_{ii},W_{jj}$ is at a variable bound, then $\alpha_{kk}\in\{0,1\}$ for either $k=i$ or $k=j$. Moreover, the right-hand sides of constraints (\ref{eq:vi1}) and (\ref{eq:vi2}) must be equal since the left-hand sides are the same and both constraints are binding in this case. Therefore, we can write:
\begin{alignat*}{2}
&U_{jj}W_{ii}+U_{ii}W_{jj}-U_{ii}U_{jj}=L_{jj}W_{ii}+L_{ii}W_{jj}-L_{ii}L_{jj}\\
\iff&(\alpha_{ii}+\alpha_{jj})[U_{ii}U_{jj}-L_{ii}U_{jj}-L_{jj}U_{ii}+L_{ii}L_{jj}]\\&=U_{ii}U_{jj}-L_{ii}U_{jj}-L_{jj}U_{ii}+L_{ii}L_{jj}\\
\iff& \alpha_{ii}+\alpha_{jj} = 1.
\end{alignat*}

Thus we have either $\alpha_{ii}=1$, in which case $W_{ii}=U_{ii},W_{jj}=L_{jj}$ or $\alpha_{jj}=1$, in which case $W_{ii}=L_{ii},W_{jj}=U_{jj}$.  First suppose $W_{ii}=U_{ii},W_{jj}=L_{jj}$.  Define the following matrix:

\[X^A := \left[\begin{array}{cc} U_{ii} & \sqrt{\frac{U_{ii}L_{jj}}{1+\alpha^2_{ij}}}\\  \sqrt{\frac{U_{ii}L_{jj}}{1+\alpha^2_{ij}}} & L_{jj}\end{array}\right] + \I \left[\begin{array}{cc} 0 & \alpha_{ij}\sqrt{\frac{U_{ii}L_{jj}}{1+\alpha^2_{ij}}}\\  -\alpha_{ij}\sqrt{\frac{U_{ii}L_{jj}}{1+\alpha^2_{ij}}} & 0 \end{array}\right],
\]
and as usual, denote the components as $X^A:= W^A + \I T^A$. By construction, we have that $(W_{ij}^A)^2+(T_{ij}^A)^2=W_{ii}^AW_{jj}^A$, so $X^A \in \mathcal{J}_C$.  We want to show that, with $X^A$, constraints~(\ref{eq:RSOC3a})-(\ref{eq:RSOC4}), (\ref{eq:vi1}), and (\ref{eq:vi2}) are binding.  Constraints~(\ref{eq:RSOC3a})-(\ref{eq:RSOC4}) can be confirmed by observation.  Recall that constraints~(\ref{eq:vi1}) and (\ref{eq:vi2}) share the same right-hand side:
\[U_{jj}W^A_{ii}+U_{ii}W^A_{jj}-U_{ii}U_{jj}=U_{ii}L_{jj}.\]

The left-hand side of constraint~(\ref{eq:vi1}) or (\ref{eq:vi2}) is:
\begin{alignat*}{2}
&\pi_0 + \pi_1 W^A_{ii} + \pi_2 W^A_{jj} + \pi_3 W^A_{ij} + \pi_4 T^A_{ij} \\
=&\pi_0 + \pi_1 W^A_{ii} + \pi_2 W^A_{jj} + (\pi_3+\alpha_{ij}\pi_4) W^A_{ij} \\
=&-\sqrt{L_{ii}L_{jj}U_{ii}U_{jj}}-\sqrt{L_{jj}U_{jj}}U_{ii}-\sqrt{L_{ii}U_{ii}}L_{jj} \\
&+(\sqrt{L_{ii}}+\sqrt{U_{ii}})(\sqrt{L_{jj}}+\sqrt{U_{jj}})\frac{1-f(L_{ij})f(U_{ij})+\alpha_{ij}(f(L_{ij})+f(U_{ij}))}{1+f(L_{ij})f(U_{ij})}
  \sqrt{\frac{U_{ii}L_{jj}}{1+\alpha^2_{ij}}}\\
=& \sqrt{U_{ii}L_{jj}} \bigg ( -\sqrt{L_{ii}U_{jj}}-\sqrt{U_{ii}U_{jj}}-\sqrt{L_{ii}L_{jj}} \\
& + (\sqrt{L_{ii}}+\sqrt{U_{ii}})(\sqrt{L_{jj}}+\sqrt{U_{jj}})\frac{1-f(L_{ij})f(U_{ij})+\alpha_{ij}(f(L_{ij})+f(U_{ij}))}{(1+f(L_{ij})f(U_{ij})) \sqrt{1+\alpha^2_{ij}}} \bigg )\\
= &U_{ii}L_{jj}.
\end{alignat*}

The last equality follows from \Cref{lem:angles}. The argument for the case $W_{ii}=L_{ii},W_{jj}=U_{jj}$ follows by symmetry.
\\ \\
\noindent\emph{Case 3}: Exactly one of the constraints~(\ref{eq:vi1}) and (\ref{eq:vi2}) is binding: \\
From Claim~\eqref{claim:wclaim} we need only consider $W_{ij} \neq 0$, so constraint~(\ref{eq:RSOC4}) can count for at most one linearly independent constraint; thus, let $T_{ij}=\alpha_{ij}W_{ij}$, where $\alpha_{ij} \in \{L_{ij},U_{ij}\}$.  This gives us at most two linearly independent linear constraints: (\ref{eq:RSOC4}) and either (\ref{eq:vi1}) or (\ref{eq:vi2}).  Thus both variable bounds~(\ref{eq:RSOC3a}) and (\ref{eq:RSOC3b}) must be binding.  In Case 2 we have already considered the possibilities that $W_{ii}=U_{ii},W_{jj}=L_{jj}$ or $W_{ii}=U_{ii},W_{jj}=L_{jj}$.  Suppose, then, that $W_{ii}=U_{ii},W_{jj}=U_{jj}$ and define the corresponding matrix:
\[
X^B := \left[\begin{array}{cc} U_{ii} & \sqrt{\frac{U_{ii}U_{jj}}{1+\alpha^2_{ij}}}\\  \sqrt{\frac{U_{ii}U_{jj}}{1+\alpha^2_{ij}}} & U_{jj}\end{array}\right] + \I \left[\begin{array}{cc} 0 & \alpha_{ij}\sqrt{\frac{U_{ii}U_{jj}}{1+\alpha^2_{ij}}}\\  -\alpha_{ij}\sqrt{\frac{U_{ii}U_{jj}}{1+\alpha^2_{ij}}} & 0 \end{array}\right],
\]
\noindent
and as usual, denote the components as $X^B:= W^B + \I T^B$. By construction, we have that $(W_{ij}^B)^2+(T_{ij}^B)^2=(W_{ii}^BW_{jj}^B)$, so $X^B \in \mathcal{J}_C$.  We want to show that, with $X^B$, constraints~(\ref{eq:RSOC3a})-(\ref{eq:RSOC4}), and (\ref{eq:vi1}) are binding.  This can be done as in Case 2, by invoking \Cref{lem:angles}:
\begin{alignat*}{2}
&\pi_0 + \pi_1 W^B_{ii} + \pi_2 W^B_{jj} + \pi_3 W^B_{ij} + \pi_4 T^B_{ij} \\
=&\pi_0 + \pi_1 W^B_{ii} + \pi_2 W^B_{jj} + (\pi_3+\alpha_{ij}\pi_4) W^B_{ij} \\
=&-\sqrt{L_{ii}L_{jj}U_{ii}U_{jj}}-\sqrt{L_{jj}U_{jj}}U_{ii}-\sqrt{L_{ii}U_{ii}}U_{jj} \\
&+(\sqrt{L_{ii}}+\sqrt{U_{ii}})(\sqrt{L_{jj}}+\sqrt{U_{jj}})\frac{1-f(L_{ij})f(U_{ij})+\alpha_{ij}(f(L_{ij})+f(U_{ij}))}{1+f(L_{ij})f(U_{ij})}  \sqrt{\frac{U_{ii}U_{jj}}{1+\alpha^2_{ij}}}\\
=&\sqrt{U_{ii}U_{jj}} \bigg ( -\sqrt{L_{ii}U_{jj}}-\sqrt{U_{ii}L_{jj}}-\sqrt{L_{ii}L_{jj}}\\
&+ (\sqrt{L_{ii}}+\sqrt{U_{ii}})(\sqrt{L_{jj}}+\sqrt{U_{jj}})\frac{1-f(L_{ij})f(U_{ij})+\alpha_{ij}(f(L_{ij})+f(U_{ij}))}{(1+f(L_{ij})f(U_{ij})) \sqrt{1+\alpha^2_{ij}}} \bigg )\\
=& U_{ii}U_{jj}\\
=& U_{jj}W^B_{ii}+U_{ii}W^B_{jj}-U_{ii}U_{jj}.
\end{alignat*}
Finally, suppose that  $W_{ii}=L_{ii},W_{jj}=L_{jj}$ and define the corresponding matrix:
\[X^C := \left[\begin{array}{cc} L_{ii} & \sqrt{\frac{L_{ii}L_{jj}}{1+\alpha^2_{ij}}}\\  \sqrt{\frac{L_{ii}L_{jj}}{1+\alpha^2_{ij}}} & L_{jj}\end{array}\right] + \I \left[\begin{array}{cc} 0 & \alpha_{ij}\sqrt{\frac{L_{ii}L_{jj}}{1+\alpha^2_{ij}}}\\  -\alpha_{ij}\sqrt{\frac{L_{ii}L_{jj}}{1+\alpha^2_{ij}}} & 0 \end{array}\right].
\]
By the same argument as used with $X^B$, which we omit to avoid repetition, $X^C$ belongs to $\mathcal{J}_C$ where constraints ~(\ref{eq:RSOC3a})-(\ref{eq:RSOC4}), and (\ref{eq:vi2}) are binding.  Thus, in all cases every extreme point belongs to $\mathcal{J}_C$. \end{proof}

\subsubsection*{Numerical Example of $\mathcal{J}_C$}
Consider the following instance of $\mathcal{J}_C$ with $i=1,j=2$:
\begin{subequations}
\label{eq:excall}
\begin{alignat}{2}
&0 \leq W_{11} \leq 1,\label{eq:exc2}\\
&1 \leq W_{22} \leq 4,\label{eq:exc3}\\
&0 \leq T_{12} \leq \frac{4}{3}W_{12}, \label{eq:exc4}\\
&W_{12}\geq 0, \label{eq:exc5}\\
&W_{11}W_{22}=W_{12}^2+T_{12}^2.\label{eq:exc1}\\
\end{alignat}
\end{subequations}

The coefficients $\pi$ of the valid inequalities~(\ref{eq:vi1})-(\ref{eq:vi2}) are
\begin{alignat*}{2}
&f(L_{12})=0,f(U_{12})=0.5,\\
&\pi_0 = 0, \pi_1 = -2, \pi_2 = 0, \pi_3 = 3,\pi_4 = 1.5.
\end{alignat*}

Valid inequality~(\ref{eq:vi1}) is
\[ -6W_{11}-W_{22}+3W_{12}+1.5T_{12}+4 \geq 0, \]
and valid inequality~(\ref{eq:vi2}) is
\[ -3W_{11}+3W_{12}+1.5T_{12} \geq 0. \]

\subsubsection*{The Real Case: $\mathcal{J}_R$}
We now consider the special case of the real variables. Let $\mathcal{J}_R$ be the set of symmetric matrices $W$ that satisfy the following constraints:
\begin{subequations}
\begin{alignat}{2}
&L_{ii} \leq W_{ii} \leq U_{ii},\label{eq:JR1}\\
&L_{jj} \leq W_{jj} \leq U_{jj},\label{eq:JR2}\\
& W_{ij} \geq 0,\label{eq:JR3}\\
&W_{ii}W_{jj}=W_{ij}^2 \label{eq:JR4}.
\end{alignat}
\end{subequations}

$\mathcal{J}_R$ can be seen as a special case of  $\mathcal{J}_C$ where $L_{ij}=U_{ij}=0$, and, therefore, we have the following corollary.

\begin{corollary}
\label{cor:realvi}
The convex hull of $\mathcal{J}_R$ can be described with the following constraints:
\end{corollary}
\begin{subequations}
\begin{alignat}{2}
& L_{ii} \leq W_{ii} \leq U_{ii}\\
& L_{jj} \leq W_{jj} \leq U_{jj}\\
& (\sqrt{L_{ii}}+\sqrt{U_{ii}})(\sqrt{L_{jj}}+\sqrt{U_{jj}})W_{ij}\geq (U_{jj}+\sqrt{L_{jj}U_{jj}})W_{ii}\nonumber \\
&+(U_{ii}+\sqrt{L_{ii}U_{ii}})W_{jj}+\sqrt{L_{ii}L_{jj}U_{ii}U_{jj}}-U_{ii}U_{jj}\label{eq:realvi1}\\
& (\sqrt{L_{ii}}+\sqrt{U_{ii}})(\sqrt{L_{jj}}+\sqrt{U_{jj}})W_{ij}\geq (L_{jj}+\sqrt{L_{jj}U_{jj}})W_{ii}\nonumber \\
&+(L_{ii}+\sqrt{L_{ii}U_{ii}})W_{jj}+\sqrt{L_{ii}L_{jj}U_{ii}U_{jj}}-L_{ii}L_{jj}\label{eq:realvi2}\\
& W \succeq 0
\end{alignat}
\end{subequations}
\begin{proof}
This is a special case of \Cref{prop:convhullj} with $L_{ij}=U_{ij}=0$, which due to constraint~(\ref{eq:RSOC4}) is equivalent to setting $T_{ij}=0$, resulting in a matrix with only real entries.  Dropping $T_{ij}$ gives the desired result.
\end{proof}

\subsubsection*{Numerical Example of $\mathcal{J}_R$}
We shall use the numerical example of $\mathcal{J}_C$ but setting $L_{12}=U_{12}=0$ and dropping $T_{12}$ accordingly. For the real case, valid inequality~(\ref{eq:vi1}) is
\[ -6W_{11}-W_{22}+3W_{12}+4 \geq 0, \]
and valid inequality~(\ref{eq:vi2}) is
\[ -3W_{11}+3W_{12} \geq 0. \]

The inequalities for the real case are shown in Figure~\ref{fig:viex}.  The feasible region of $\mathcal{J}_R$ is depicted in the upper-left quadrant. Black spheres are centered around the intersection points of variable bounds on the cone. The upper-right quadrant depicts the intersection of $W_{12}^2=W_{11}W_{22}$ with the variable bounds at $W_{11}=1$ and $W_{22} = 1$.  The lower-right quadrant depicts the cone with valid inequality~(\ref{eq:vi1}). The intersection is ellipsoidal and three of the highlighted points lie on the boundary of the valid inequality ($\{W_{11}=0, W_{22}=4, W_{12} = 0,\},\{W_{11}=1, W_{22}=1, W_{12} = 1\},\{W_{11}=1, W_{22}=4, W_{12} = 2\}$).  The lower-left quadrant depicts the intersection of the cone with valid inequality~(\ref{eq:vi2}).  The intersection is a hyperbola, and three highlighted points lie on the boundary of the valid inequality ($\{W_{11}=0, W_{22}=1, W_{12} = 0,\},\{W_{11}=0, W_{22}=4, W_{12} = 0\},\{W_{11}=1, W_{22}=1, W_{12} = 1\}$).

\begin{figure}[h!]
  \centering
    \includegraphics[width=0.8\textwidth]{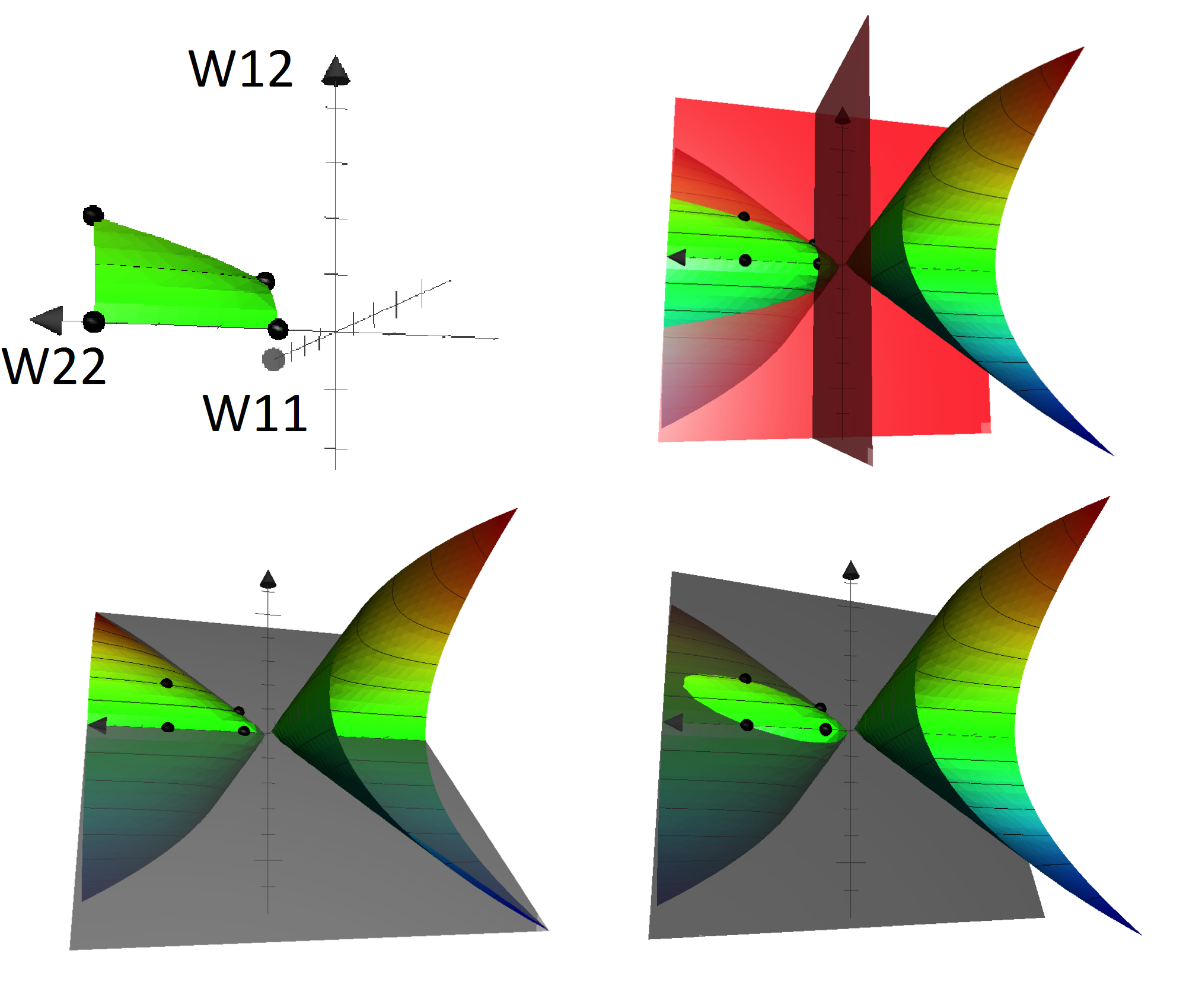}
      \caption{Clockwise, starting from top-left: $\mathcal{J}_R$; cone of~(\ref{eq:exc1}), upper bound of~(\ref{eq:exc2}), and lower bound of~(\ref{eq:exc3}); cone of~(\ref{eq:exc1}) and valid inequality~(\ref{eq:vi1}); cone of~(\ref{eq:exc1}) and valid inequality~(\ref{eq:vi2}).}
\label{fig:viex}
\end{figure}

\subsubsection*{Comparison with RLT Inequalities}
Consider the RLT inequalities:
\begin{subequations}
\begin{alignat}{2}
& W_{ij} \leq u_jw_i + \ell_i w_j - \ell_iu_j, \label{eq:RLT1}\\
\bold{(RLT)} \ \ \  & W_{ij} \leq \ell_jw_i + u_iw_j - \ell_ju_i,\label{eq:RLT2}\\
& W_{ij} \geq \ell_jw_i + \ell_iw_j - \ell_i\ell_j, \label{eq:RLT3}\\
& W_{ij} \geq u_jw_i + u_iw_j - u_iu_j. \label{eq:RLT4}
\end{alignat}
\end{subequations}

Note that for RLT we allow the possibility that $i=j$.  Inequalities \eqref{eq:RLT1}--\eqref{eq:RLT4} are derived from the RLT procedure of Sherali and Adams \cite{sherali1992global}. They are also known as McCormick estimators as they can be derived from earlier work on convex envelopes by McCormick \cite{mccormick1976computability}. For $w \in \mathbb{R}^2$, Anstreicher and Burer \cite{anstreicher2010computable} prove that the convex hull of the real bounded set $\{(W,w): \ W=ww^T, \ell \leq w \leq u\}$ is given by the RLT inequalities on $W_{ii},W_{jj}$ and $W_{ij}$ together with the SDP constraint $W \succeq ww^T$.  As discussed, valid bounds for $\mathcal{J}_C$ can be derived from the complex  bounded set $l \leq x \leq u$ for $x \in \mathbb{C}^2$.  We will do a thorough comparison at the end of this section, but let us first proceed to analyzing the convex hull of $\mathcal{J}_C$.

$\mathcal{J}_R$ is implied by the following constraints:

\begin{subequations}
\begin{alignat}{2}
&\sqrt{L_{ii}}\leq w_i \leq \sqrt{U_{ii}},\label{eq:rbox1}\\
&\sqrt{L_{jj}}\leq w_j \leq \sqrt{U_{jj}},\label{eq:rbox2}\\
& W = ww^T.\label{eq:rbox3}
\end{alignat}
\end{subequations}

From Burer and Anstreicher \cite{anstreicher2010computable} we know the convex hull of constraints~({\ref{eq:rbox1})-({\ref{eq:rbox3}) can be described by the RLT inequalities on  $W_{ii},W_{jj}$ and $W_{ij}$ together with the convex constraint $W \succeq ww^T$.  Constraints~({\ref{eq:rbox1})-({\ref{eq:rbox3}) can be rewritten as:

\begin{subequations}
\begin{alignat}{2}
&L_{ii} \leq W_{ii} \leq U_{ii},\\
&L_{jj} \leq W_{jj} \leq U_{jj},\\
& W =ww^T.
\end{alignat}
\end{subequations}

Finally, $W=ww^T$ for a vector of free variables $w$ holds iff $W$ is positive semidefinite and of rank one, which gives equivalence in the space of $W$ between $\mathcal{J}_R$ and the set of feasible solutions to constraints~({\ref{eq:rbox1})-({\ref{eq:rbox3}).  Hence the RLT inequalities on $W_{ii},W_{jj},W_{ij}$ dominate valid inequalities~(\ref{eq:vi1}) and (\ref{eq:vi2}) in the special case of the reals as they describe the convex hull of the same set in the lifted space that includes $w$ variables.

We note that the inequalities coincide on diagonal entries of $W$.  The RLT inequalities for diagonal matrix entries are

\begin{subequations}
\begin{alignat}{2}
& W_{ii} \leq u_iw_i + \ell_i w_i - \ell_iu_i,\label{eq:RLTd1}\\
& W_{ii} \geq 2\ell_iw_i - \ell_i^2, \label{eq:RLTd2}\\
& W_{ii} \geq 2u_iw_i - u_i^2. \label{eq:RLTd3}
\end{alignat}
\end{subequations}

For $i>1$ ($i=1$ merely gives $W_{11}=1$) inequality~(\ref{eq:RLTd1}) can also be obtained by applying either (\ref{eq:realvi1}) or (\ref{eq:realvi2}) to $\left[\begin{array}{cc}
1 & w_i\\
w_i & W_{ii}\end{array}\right]$, which is the $2\times 2$ principal submatrix of $Y$ with indices $1$ and $i$.  The remaining inequalities~(\ref{eq:RLTd2})-(\ref{eq:RLTd3}) are implied by the SDP constraint since $(w_i-l\ell_i)^2\geq 0,(w_i-u_i)^2\geq 0 \implies w_i^2 \geq \max \{2\ell_iw_i - \ell_i^2, 2u_iw_i - u_i^2 \}$ and $W_{ii} \geq w_i^2$.

In the complex case $\mathcal{J}_C$, we can show that domination runs the other direction: a natural application of RLT does not capture the convex hull of $\mathcal{J}_C$. One possible transformation to RQCQP involves a matrix of the form
\[ \left(\begin{array}{cc}
1\\ w \\ t \end{array}\right)\left(\begin{array}{cc}
1\\ w \\ t \end{array}\right)' ,\]
where the components of the complex vector $x:= w + \I t$ are treated as separate decision variables. In this case, the RLT inequalities together with the SDP constraint~(\ref{eq:SSOC1}) do not give the convex hull of $\mathcal{J}_C$ as shown in the example below.

We continue with the numerical example of $\mathcal{J}_C$, and we will show that there is a point satisfying the positive semidefinite condition and the RLT inequalities, but is outside $\mbox{conv}(\mathcal{J}_C)$.  Thus an application of RLT inequalities to the RQCQP reformulation is not sufficient to describe $\mbox{conv}(\mathcal{J}_C)$.  Let us add the complex variables $x_1,x_2$ with the following bounds on magnitude:
\begin{alignat*}{2}
&0 \leq |x_{1}| \leq 1,\\
&1 \leq |x_{2}| \leq 2.\\
\end{alignat*}

Considering the real and imaginary components of $x$ as separate decision variables, let us transform the problem into RQCQP.  Moreover, we require bounds on $\mbox{Re}(x),\mbox{Im}(x)$.  We shall use the magnitude-implied bounds:

\begin{alignat*}{2}
& |w_1| \leq 1, \ |t_1| \leq 1,\\
& |w_2| \leq 2, \ |t_2| \leq 2.
\end{alignat*}

Now $W_{12} = \mbox{Re}(x_1x_2^*) = w_1w_2+t_1t_2$, and $T_{12} = \mbox{Im}(x_1x_2^*) = t_1w_2-w_1t_2$.  Applying RLT inequalities to each bilinear term, we obtain the following inequalities:
\begin{subequations}
\label{eq:RLTex}
\begin{alignat}{2}
\label{eq:real-rlt-first} &W_{12} \leq -|2w_1-w_2|-|2t_1-t_2|+4\\
&W_{12} \geq |2w_1+w_2|+|2t_1+t_2|-4\\
&T_{12} \leq -|2t_1-w_2|-|2w_1+t_2|+4\\
&T_{12} \geq |2t_1+w_2|+ |2w_1-t_2|-4\\
& W_{11} \leq 2,\\
& W_{11} \geq 2|w_1|+2|t_1|-2,\\
& W_{22} \leq 8,\\
\label{eq:real-rlt-last} & W_{22} \geq 4|w_1|+4|t_1|-8.
\end{alignat}
\end{subequations}

RLT inequalities~(\ref{eq:real-rlt-first})--(\ref{eq:real-rlt-last}) admit the solution $x=0,T_{12}=0,W_{12}=0,W_{11}=1,W_{22}=4$, which satisfies constraints~(\ref{eq:exc2})-(\ref{eq:exc4}) and the positive semidefinite constraint. However, this solution violates valid inequality~(\ref{eq:vi2}), where $W_{12}=T_{12}=0$ implying $W_{11}=0$.

\subsubsection*{Constructing Bound Matrices $L$ and $U$}
We provide a procedure to derive valid bounds so that $\mathcal{J}_C$ is applicable to any CQCQP.  For certain problems the procedure may be unnecessary: values might be obtained from the problem directly or via bound tightening procedures. For instance, in ACOPF, constraints~(\ref{eq:RSOC3a}) and (\ref{eq:RSOC3b}) are given as voltage magnitude bounds and constraint~(\ref{eq:RSOC4}) models nodal phase angle difference bounds.

\subsubsection*{Constraints~(\ref{eq:RSOC3a}) and (\ref{eq:RSOC3b})}
$W_{ii}= |x_{i-1}|$ for $i \geq 2$, representing constraints on CQCQP variable magnitudes, and likewise for $j$.   As magnitudes are nonnegative, a valid lower bound can always be obtained by setting $L_{ii} = L_{jj} = 0$.  Furthermore, upper bounds can be obtained directly from the original variable bounds~(\ref{eq:vbounds}).  For $W_{11}$ we can set $L_{11}=U_{11}=1$.

\subsubsection*{Constraints~(\ref{eq:RSOC4}) and (\ref{eq:RSOCnn}}
A sufficient condition to derive valid bounds is that either $w$ or $t$ (or both) have strictly positive entries. This can be done using an affine shift.  Observe that for any $\ell$ we have:
\[x^*Qx + c^*x + b = (x-\ell)^*Q(x-\ell) + (c^*+2\ell^*Q)(x-\ell) + b +\ell^*Q\ell+c^*\ell.\]

Therefore with substitution of variables $q := x-\ell+e+\I e$, where $e$ is the ones vector, any (bounded) complex QCQP may be rewritten with a decision vector with only positive components.

Since $Y=yy^*$ in the lifted formulation of CQCQP we have $W_{ij}=w_iw_j+t_it_j$ and $T_{ij}=t_iw_j-w_it_j$. Thus we can assume $0< W^-\leq W_{ij}$, where $W^- := w^L_{i}w^L_{j}+t^L_{i}t^L_{j}$.  From $\mbox{rank}(Y)=1$ we have that $W_{ij}^2+T_{ij}^2=W_{ii}W_{jj}$, and so:
\begin{alignat*}{2}
\sqrt{\frac{U_{ii}U_{jj}}{(W^-)^2}-1}&\geq&\sqrt{\frac{W_{ii}W_{jj}}{W_{ij}^2}-1} = \frac{\left|T_{ij}\right|}{W_{ij}}.
\end{alignat*}
Hence we have valid bounds for constraint~(\ref{eq:RSOC4}): $-L_{ij}=U_{ij} = \sqrt{U_{ii}U_{jj}/(W^-)^2-1}$.  Moreover, by construction $W_{ij} \geq 0$ which implies constraint~(\ref{eq:RSOCnn}).

Note that constraints describing $\mathcal{J}_C$ shifted back to the original variables will include the original variables of CQCQP.  For instance, suppose we set $q=x-l+e+\I e$. Moreover, define the components $q := w^q + \I t^q$.  Then we have
\begin{alignat*}{2}
W^q_{ij}&= (w_i-w^L_{i}+1)(w_j-w^L_j+1)+(t_i-t^L_{i}+1)(t_j-t^L_j+1)\\
&=w_iw_j+t_it_j+(1-w_j^L)w_i+(1-w_i^L)w_j+(1-w_i^L)(1-w_j^L)\\
&+(1-t_j^L)t_i+(1-_i^L)t_j+(1-t_i^L)(1-t_j^L)\\
&= W_{ij} + (1-w_j^L)w_i+(1-w_i^L)w_j+(1-t_j^L)t_i+(1-_i^L)t_j\\
&+(1-w_i^L)(1-w_j^L)+(1-t_i^L)(1-t_j^L).
\end{alignat*}

Thus a valid inequality derived for $W^q_{ij}$ gives a valid inequality for $W_{ij},x_i,x_j$.

In the special case of real QCQP, we can instead set $L_{ij} = U_{ij} = 0$, which enforces the matrix of imaginary entries $T$ to be the zeros matrix. An affine transformation to a problem with nonnegative variables is sufficient to ensure that constraint~(\ref{eq:RSOCnn}) is valid.

\subsubsection*{Constraint~(\ref{eq:RSOC1})}
This constraint is implied by $Y=yy^*$, which requires $Y$ to have rank one.

\subsection{Branching on a Complex Matrix Entry}  In this section we consider branching on upper and lower bounds of $Y$ as specified in $\mathcal{J}_C$, i.e., bounds on every complex matrix entry.  We examine branching rules for selecting a single $(i,j)$ entry of $Y$ to branch on.  One way to form a branching rule is to use a scoring function, where the branching option with the highest score is selected.  A score can be based on the violation of relaxed constraints, or an estimate of the impact of branching on the children nodes' optimal relaxation objective value.  In the standard development of the SDP relaxation the only relaxed constraint is the rank-one constraint on $Y$.  However, the rank function is discrete and applies globally to all variables of the decision matrix, which seems problematic for use in variable branching.  Therefore, we consider an alternative to the rank-one condition:

\begin{proposition}
\label{prop:prinrank}
For $n>1$ a nonzero Hermitian positive semidefinite $n\times n$ matrix $Y$ has rank one iff all of its $2\times 2$ principal minors are zero.
\end{proposition}

\begin{proof}
Suppose $Y$ has rank $r > 1$.  Since $Y$ is Hermitian it has an $r\times r$ nonzero principal minor. Since $Y$ is positive semidefinite this principal minor corresponds to a positive definite $r\times r$ submatrix.  As $r\geq 2$, this implies there exists a $2\times 2$ strictly positive principal minor.   Now suppose instead that $Y$ has a strictly positive $2\times 2$  principal minor. Then $Y$ contains a rank-two principal submatrix and thus $r>1$.
\end{proof}
We will use the equivalent condition that the minimum eigenvalue of each $2\times 2$ principal submatrix of $Y$ be zero.  Recall that by the Hermitian property of $Y$ we have for any $i,j$ that $W_{ij}=W_{ji},T_{ij}=-T_{ji}$ and thus $T_{ii}=0$. Algebraically the minimum eigenvalue condition can be expressed as
\[\lambda_{\min} = \frac{1}{2} \bigg ( W_{ii} + W_{jj} - \|(W_{ii}-W_{jj},2W_{ij},2T_{ij})\| \bigg ). \]

We branch by partitioning the range $[L_{ij}, U_{ij}]$ for some $(i,j)$ via updating the bounds as:
\[L_{ij}' \leftarrow \alpha L_{ij}+(1-\alpha)U_{ij}\]
\[U_{ij}' \leftarrow \alpha L_{ij}+(1-\alpha)U_{ij}\]

\noindent
where $\alpha \in (0,1)$ is a parameter.  In our implementation we use the bisection rule, i.e., set $\alpha = 0.5$. We will refer to the assignment $L_{ij}' \leftarrow \frac{L_{ij}+U_{ij}}{2}$ as the up branch with respect to a matrix entry, and similarly down branch refers to the assignment on the upper bound.  Now let us consider rules for selecting an entry to branch on.

\subsubsection*{Most Violated with Strong Branching (MVSB)}
Let $c$ be a pair of indices $(i,j)$ for $i\ neq j$. From \Cref{prop:prinrank} it follows that for CSDP $\mbox{rank}(Y)\in \{0,1\}$ iff $\lambda_{\min} (Y_c )=0$ for all pairs. Select the $2\times 2$ principal submatrix of $Y$ with the greatest minimum eigenvalue, and let $c^*$ be the indices of this submatrix.  Given $c^*$, there are three possible candidate entries.  These are evaluated by strong branching: for each entry we will solve the up branch problem and obtain the solution matrix $Y^+$, and likewise $Y^-$ from the down branch. The following score function is used:
 \[\mu\max\{-\lambda_{\min}(Y_{c^*}^+),-\lambda_{\min}(Y_{c^*}^-)\}+(1-\mu)\min\{-\lambda_{\min}(Y_{c^*}^+),-\lambda_{\min}(Y_{c^*}^-)\}, \]

\noindent
where $\mu \in [0,1]$ is a tuning parameter; we follow the example of COUENNE  \cite{belotti2013mixed} and use a value of 0.15.  The entry with the highest score is selected for branching.

\subsubsection*{Most Violated with Worst-Case Bounds (MVWB)}
Since strong branching is computationally expensive, we consider solving a simpler subproblem to produce a score.  Consider the Worst-Case Eigenvalue (WEV) problem of finding the greatest minimum eigenvalue that can be obtained within $\mbox{conv}(\mathcal{J}_C)$:
\[\bold{(WEV)} \ \max \lambda \
\mbox{s.t.: \ } \|(W_{ii}-W_{jj},2W_{ij},2T_{ij})\|\leq W_{ii} + W_{jj} - 2\lambda, \
(\ref{eq:RSOC3a})-(\ref{eq:RSOC4}), (\ref{eq:vi1})-(\ref{eq:vi2}). \]

Note that we dropped the positive semidefinite condition on $X$ since the objective maximizes the minimum eigenvalue. We solve WEV in lieu of solving the children problems $Y^+$ and $Y^-$. Thus, overestimates of $\lambda_{\min}(Y_{c^*}^+)$ and $\lambda_{\min}(Y_{c^*}^-)$ are used in the score function. MVWB is otherwise the same as MVSB.

\subsubsection*{Reliability Branching with Entry Bounds (RBEB)}
Since MVSB and MVWB rely on a particular violation metric, for benchmarking purposes we consider a method that is agnostic to the measure of violation.  RBEB is an application of the \textit{rb-int-br} rule of Belotti et al. \cite{belotti2013mixed}.

Reliability branching uses pseudocosts, which capture information about previous branching decisions.  Let $\Phi^+$ and $\Phi^-$ be matrices containing pseudocosts, estimating the improvement in objective value by branching up or down, respectively.  If, at search tree node $k$, $L^k_{ij}$ is selected for branching up and the objective improves by $\delta^k$ in the child node's relaxation, then let $D^k:=\delta^k/(\frac{U^k_{ij}-L^k_{ij}}{2}-L^k_{ij})$ be the per-unit improvement.  $\Phi^+$ is the running average of all $D^k$ for up branches, and $\Phi^-$ is the running average for down branches.

For a given candidate $(i,j)$, the following score function is used:

 \[\left(\mu\max\{\Phi^+_{ij},\Phi^-_{ij}\}+(1-\mu)\min\{\Phi^+_{ij},\Phi^-_{ij}\}\right)\left(\frac{U_{ij}-L_{ij}}{2} \right).\]

As with MVSB and MVWB, we use a value of $\mu=0.15$.  The candidate with the highest score is selected for branching.  Note that we can restrict the set of candidate entries to those with violation, i.e. corresponding to members of $2\times 2$ principal submatrices with strictly positive minimum eigenvalue.

In reliability branching, strong branching is used in lieu of pseudocosts until $\eta$ evaluations have been performed on a given up or down branch; $\eta$ is called the reliability parameter  \cite{achterberg2005branching}.  In computational experiments we report results for $\eta=1$; $\eta=4$ gave similar but poorer performance due to the computational burden of strong branching on large instances.

\subsection{Bound Tightening}
It is well recognized that reducing the domain of a problem by tightening the bounds of the variables can improve the performance of branching methods substantially.
In this subsection we describe two fast procedures to tighten the bounds $L_{ij},U_{ij}$. In a companion applications paper \cite{cao2015acopf}, we apply these procedures along with others that exploit the special structure of the ACOPF problem. 

\subsubsection*{Tightening with a quadratic inequality}

Let us consider the following two variable problem:
\begin{alignat*}{2}
aq^2+qy+c & \leq 0,\\
\ell_y \leq y & \leq u_y.
\end{alignat*}

\noindent
Here $a$ and $c$ are parameters, and $q$ and $y$ are real-valued variables.  Given variable bounds on $y$, we would like to infer variable bounds on $q$.  With appropriate transformations, a variety of quadratic constraints of CQCQP can be put in the simple form above. For instance, one can convert a complex quadratic inequality into a real quadratic inequality in terms of the real and imaginary variable components.  Now any real quadratic constraint $q^TDq+b^Tq+c\leq 0$ may be written in the form $d_{ii}q_i^2 + q_i\sum_{j\neq i}d_{ij}q_j + \sum_{j\neq i}d_{jj}q_j^2 + b^Tq + c \leq 0$.  If all variables are bounded, then $\sum_{j\neq i}d_{ij}q_j$ may be treated as a single bounded variable and $b^Tq + c$ can be fixed to its minimum value, which allows us to apply the proposed bound tightening structure. One can also apply the same principles to derive bounds on the magnitude of $x_i$, and thereby derive bounds for $W_{ii}$ (see \cite{cao2015acopf} for details).

By adding a slack variable $s$, let us write the two variable quadratic constraint as an equality:
\begin{alignat*}{2}
 aq^2+qy+c+s &  =0,\\
s  & \geq 0,\\
\ell_y \leq y & \leq u_y.
\end{alignat*}

\noindent
Solving the quadratic equation for $q$ we have:
\begin{equation}
\label{eq:quadroot}
q = \frac{-y \pm \sqrt{y^2-4ac-4as}}{2a}.
\end{equation}
From here, it is possible to find the maximum and minimum values of the right-hand side of equality~(\ref{eq:quadroot}) with respect to $y,s$, given their bounds.  We can thus infer lower and upper bounds on $q$; and if the lower bound is greater than the upper bound, then infeasibility can be inferred.

\subsubsection*{Tightening on cycles}
We can tighten the off-diagonal entry bounds, $L_{ij},U_{ij}, i \neq j$ based on the simple principle that the sum of differences around a cycle must equal to zero.  Denote the difference of some variables of $x$, $\delta_{ij}:= x_i - x_j$. Given some cycle of indices, say $\{1,2,3\}$, we have $\delta_{12}+\delta_{23}+\delta_{31}=0$.

To interpret the off-diagonal terms as difference of variables, it will be convenient to reformulate CQCQP in polar coordinates.  For any complex variable $x_i$, we may replace the real and imaginary components with the complex angle $\theta_i$ and the magnitude $|x_i|$, where $\mbox{Re}(x_i) = |x_i|\cos(\theta_i), \mbox{Im}(x_i) = |x_i|\sin(\theta_i)$.  Then we have:
\begin{alignat*}{2}
W_{ij} & = \mbox{Re}(x_i x_j) = |x_i||x_j|\cos(\theta_i-\theta_j),\\
T_{ij} & =\mbox{Im}(x_i x_j) = |x_i||x_j|\sin(\theta_i-\theta_j).
\end{alignat*}

\noindent
Therefore, constraint~(\ref{eq:RSOC4}) implies the following:
\begin{alignat*}{2}
&L_{ij} W_{ij} \leq T_{ij} \leq  U_{ij}W_{ij}\\
\implies & L_{ij}|x_i||x_j|\cos(\theta_i-\theta_j) \leq |x_i||x_j|\sin(\theta_i-\theta_j) \leq U_{ij}|x_i||x_j|\cos(\theta_i-\theta_j)\\
\implies & L_{ij} \leq \tan(\theta_i-\theta_j) \leq U_{ij}.\\
\end{alignat*}

The last implication is due to the implicit nonnegativity of $W_{ij}$ per constraint~(\ref{eq:RSOC4}).  Using the fact that the arctangent is an increasing function, we  can now apply the cycle rule to infer new bounds by fixing all but one variable at variable bounds. For instance, for the cycle $\{1,2,3\}$ we can see if $L_{12}$ can be tightened:
\begin{alignat*}{2}
&\theta_1-\theta_2+\theta_2-\theta_3+\theta_3-\theta_1=0\\
\implies& \arctan(L_{12}) +\arctan(U_{23})+ \arctan(U_{31}) \geq 0.
\end{alignat*}

For instance, if $U_{23}=0.25, U_{31}= 0.5$, then $L_{12} \geq -\tan(\arctan(0.5)+\arctan(0.25))=-\frac{6}{7}$.

\section{Computational Results and Analysis}
\label{sec:experiments}
In this section we present the results of experiments on solving ACOPF and BoxQP problems using the spacial branch-and-cut approach described in the previous section. The experiments test the effect of using different relaxations --- standard SDP, SDP+RLT inequalities~\eqref{eq:RLT1}--\eqref{eq:RLT4}, SDP+complex valid inequalities~(\ref{eq:vi1})--(\ref{eq:vi2})) --- as well as the proposed branching rules MVSB, MVWB and the benchmark reliability rule RBEB on both RQCQP and CQCQP instances.  For ACOPF RLT inequalities were generated with the same complex-to-reals transformation as used for inequalities~(\ref{eq:real-rlt-first})--(\ref{eq:real-rlt-last}). Results are summarized in this paper, and instance-specific data is provided at \url{https://sites.google.com/site/cchenresearch/}.

All experiments herein are conducted with a 3.2 GHz quad-core Intel i5-4460 CPU processor and 8 GB main memory. Algorithms are implemented using MATLAB \cite{guide1998mathworks} with model processing performed by YALMIP \cite{lofberg2004yalmip}. Conic programs are solved with MOSEK version 7.1 \cite{andersen2000mosek}. IPOPT version 3.11.1 \cite{wachter2006implementation} is used as a local solver to obtain primal feasible solutions to CQCQP at each search tree node.

All spatial branch-and-cut (SBC) configurations are implemented with a depth-first search node selection rule.  The search termination criteria are: an explored nodes limit of 10000, a time limit of 1.5 hours, and a relative optimality gap limit. A search tree depth of 100 is applied, pruning all children nodes past this limit.  The optimality gap is calculated using the global upper bound (\texttt{gub}) and global lower bound $(\texttt{glb}): \texttt{gap} = 1-(\texttt{gub}-\texttt{glb})/|\texttt{gub}|$.

\subsection{Problem Instances}

\subsubsection*{ACOPF}
The ACOPF problem is a power generation scheduling problem that can be formulated as CQCQP. The problem formulation can be found in \Cref{sec:acopf}; for a thorough treatment on optimization issues related to optimal power flow we refer the reader to Bienstock \cite{bienstock2016electrical}.

Our experiments include the test cases of Gopalakrishnan et al. \cite{gopalakrishnan2012global}. Small duality gaps were reported for these cases, so the root relaxation is known to provide a good lower bound.   These instances are named g9, g14, g30, and g57, where the number indicates the number of buses in the problem. We include the modified IEEE test cases from Chen et al. \cite{cao2015acopf}, which are named 9Na, 9Nb, 14S, 14P, and 118IN. We also use the NESTA instances \cite{coffrin2014nesta} listed in \Cref{tab:nestsumm}. Here \texttt{SDP gap} indicates the percentage gap between the standard SDP relaxation and the best known upper bound. NESTA instances with trivial gap ($<0.1\%$) or with more than 1000 buses are excluded.  The latter criterion excludes two large instances that are too challenging for the SDP solver to handle even at the root node.

Since the standard IEEE test cases do not include phase angle differences, we have applied a 30 degree bound across all connected buses if not otherwise specified. We also use a sparse formulation of CSDP that replaces the PSD constraint with multiple positive semidefinite constraints on submatrices and linear equality constraints.  Fukuda et al. \cite{fukuda2001exploiting}  developed the methodology for sparse formulation of a generic SDP, and several authors \cite{molzahnimplementation_2013,jabr2012exploiting,bai2011chordal,bienstock2015lp} have studied its effects in improving solution times for ACOPF.  We follow this methodology, using a minimum-degree ordering heuristic on the admittance matrix and finding a corresponding symbolic Cholesky decomposition in order to determine a suitable clique decomposition.  In \Cref{sec:suff} we show that enforcing the rank constraint on each submatrix is sufficient to ensure equivalence between CQCQP and the sparse version of CSDP.  Hence valid inequalities and branching rules are applied only to $2\times 2$ principal minors that have been kept after sparse decomposition.

All instances were solved using an optimality criterion of $0.1\%$. Bound tightening is activated on all instances (see \cite{cao2015acopf} for its effects).

\begin{table}[]
\centering
\caption{NESTA instances.}
\label{tab:nestsumm}
\begin{tabular}{cc}
\hline
\hline
\texttt{name} & \texttt{SDP gap $(\%)$} \\
\hline
case3\_lmbd & 0.39 \\
case3\_lmbd\_\_api & 1.26 \\
case3\_lmbd\_\_sad & 2.06 \\
case5\_pjm & 5.22 \\
case24\_ieee\_rts\_\_api & 1.45 \\
case24\_ieee\_rts\_\_sad & 6.05 \\
case29\_edin\_\_sad & 28.44 \\
case73\_ieee\_rts\_\_api & 4.1 \\
case73\_ieee\_rts\_\_sad & 4.29 \\
case30\_as\_\_sad & 0.47 \\
case30\_fsr\_\_api & 11.06 \\
case89\_pegase\_\_api & 18.11 \\
case118\_ieee\_\_api & 31.5 \\
case118\_ieee\_\_sad & 7.57 \\
case162\_ieee\_dtc & 1.08 \\
case162\_ieee\_dtc\_\_api & 0.85 \\
case162\_ieee\_dtc\_\_sad & 3.65 \\
case189\_edin\_\_sad & 1.2\\
\hline
\hline
\end{tabular}
\end{table}

\subsubsection*{BoxQP}
 The BoxQP problem is formulated as
$\ \min \frac{1}{2}x'Qx + f'x: 0 \leq x\leq 1,$
where $x\in \mathbb{R}^n$ is decision vector, and $f \in \mathbb{R}^n,Q\in \mathbb{R}^{n\times n}$ are data. We use the 90 BoxQP instances of Burer and Vandenbussche \cite{burer2009globally}. The instances are named \emph{sparAAA-BBB-C}, where \emph{AAA} is the dimension of $x$, \emph{BBB} is the density of $Q$, and \emph{C} is the random seed number.  We set an optimality gap limit of $0.01\%$ for these instances.  The relaxation used for these instances is the SDP relaxation strengthened with RLT inequalities.

\subsection{Results}

\Cref{tab:opfsumm} shows average the performance of different convex relaxations and branching rules for the ACOPF; the averages are taken over instances solved by a given configuration.  The first column \texttt{relax} shows the relaxation used to solve the problem: CVI refers to CSDP together with complex valid inequalities~(\ref{eq:vi1})-(\ref{eq:vi2}); RLT refers to the SDP relaxation of the real QCQP formulation together with RLT inequalities; and SDP refers the standard SDP relaxation without additional valid inequalities. The column \texttt{rule} shows the branching rule used, \texttt{nodes} the number of search tree nodes explored before termination,  \texttt{depth}  the maximum search tree depth, \texttt{lbtime} the time in seconds spent solving relaxations, \texttt{ubtime} the time in seconds spent obtaining primal solutions, \texttt{time} the total time spent in seconds by the SBC algorithm, and finally \texttt{solved} the total proportion of solved instances.

The formulation CVI with the complex valid inequalities~(\ref{eq:vi1})-(\ref{eq:vi2}) leads to the best performance. The new branching rules MVWB and MVSB designed for rank-one constraints perform significantly better compared to the benchmark reliability rule RBEB.
Without using the complex valid inequalities only two of the 26 instances are solved with the RLT formulation and MVWB branching. The plain SDP formulation without using any cuts do not converge for any of the instances.
Only the best rule MVWB is presented for relaxations RLT and SDP for brevity; similar or worse results are obtained with MVSB and RBEB.

For a more detailed view, the computational results are also summarized with performance profiles \cite{petriu2002applying}; we use the standard $\log 2$ base on the $x$ axis. The graph represents the proportion of instances for which a given configuration was within $x$ times the best configuration. For instance, for a profile of time to solve, a point $(x,y)=(2,0.5)$ indicates that for $50\%$ of cases (the $y$ axis) a spatial branch-and-cut configuration achieves at most two times the best time (the $x$ axis) among all configurations tested.

For ACOPF we present performance profiles for time to solve in \Cref{fig:opftime} and total search tree nodes in \Cref{fig:opfnodes}.  Using CVI, MVWB shows consistently better results over the other branching rules, and both MVWB and MVSB compare favourably to RBEB.  MVSB used smaller search tree when it converged, although MVWB resulted in more convergent instances.  Both violation strategies tend to produce smaller search trees compared to the reliability rule.  The complex relaxation CVI produces substantially better results than the RQCQP approach represented by RLT. The addition of the valid inequalities led to faster convergence, whereas with a plain SDP relaxation no practical convergence is observed. These results clearly demonstrate the benefits of exploiting the structure in the complex formulation. The stronger complex inequalities and the new branching rules based on the violation of the rank-one constraints lead to a faster convergence.

\begin{table}[h!]
\centering
\caption{ACOPF performance summary.}
\label{tab:opfsumm}
\begin{tabular}{cc|ccccccc}
\hline
\hline
\texttt{relax} & \texttt{rule}& \texttt{nodes} & \texttt{depth} & \texttt{lbtime} & \texttt{ubtime} & \texttt{time} &  \texttt{solved} \\
\hline
CVI&MVWB & 1459.5 & 33.5 & 107.4 & 169.5 & 382.1 & 15/26 \\
CVI&MVSB & 804.6 & 21.3 & 128 & 149 & 734.3 & 12/26 \\
CVI&RBEB & 4450.8 & 62.4 & 555.9 & 1140.7 & 2077.5 & 7/26\\
RLT&MVWB & 94 & 17.5 & 2 & 6 & 9 & 2/26\\
SDP&MVWB & - & - & - & - & - & 0/26\\
\hline
\hline
\end{tabular}
\end{table}

\begin{figure}[h!]
  \centering
    \includegraphics[width=0.9\textwidth]{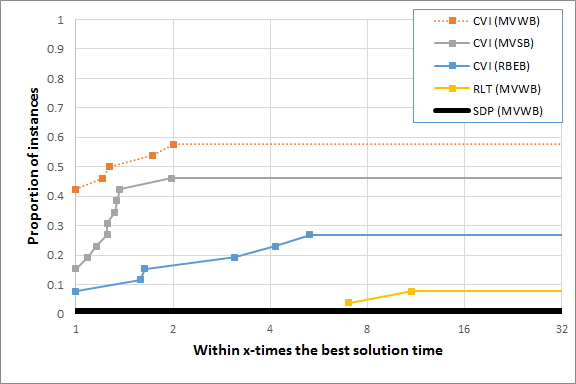}
      \caption{Performance profile of solve times using different relaxations and branching rules for ACOPF.}
\label{fig:opftime}
\end{figure}
\begin{figure}[h!]
  \centering
    \includegraphics[width=0.9\textwidth]{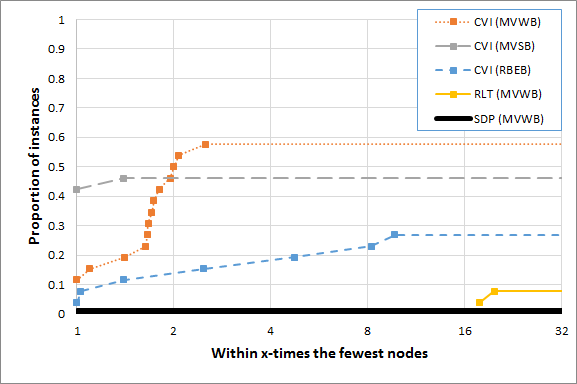}
      \caption{Performance profile of search tree nodes using different relaxations and branching rules for ACOPF.}
\label{fig:opfnodes}
\end{figure}


\Cref{tab:bqpsumm} shows a comparison of the branching rules for BoxQP. For BoxQP the lower bound times are substantially higher than the upper bound times due to the quadratic increase in variables in the lifted relaxation. The proportion of lower and upper bound solve times to total time indicates the overhead cost of strong branching: RBEB ($16\%$) MVSB ($27\%$), MVWB ($>99\%$). With the new branching rules over 80 of the 90 instances are solved; whereas with the reliability rule RBEB 72 instances are solved.

The performance profiles provide a more detailed view on the BoxQP results.  \Cref{fig:qptime} shows that MVWB was the best method in terms of solution times, and both MVWB and MVSB performed substantially better than the reliability rule RBEB.  \Cref{fig:qpnodes} indicates that MVSB tends to produce the smallest search trees, demonstrating the power of the strong branching. Both MVWB and MVSB lead to significantly smaller search trees compared to the reliability rule RBEB.
The results on BoxQP instances demonstrate that even for real QCQPs the branching rules MVWB and MVSB exploiting the rank-one constraint can be very effective.  The results are competitive compared to published results of Misener et al. \cite{misener2015dynamically}, which report results from global solvers GloMIQO \cite{misener2013glomiqo}, BARON \cite{sahinidis1996baron}, and Couenne \cite{belotti2009branching} are compared. Burer and Vandenbussche \cite{burer2009globally} developed an exact SDP-based branch-and-bound method for nonconvex quadratic programming problems based on KKT conditions.  Substantial improvements in performance based on completely positive programming are reported by Burer \cite{burer2010optimizing} and Burer and Chen \cite{chen2012globally}.

\begin{table}[]
\centering
\caption{BoxQP performance summary.}
\label{tab:bqpsumm}
\begin{tabular}{c|ccccccc}
\hline
\hline
\texttt{rule}& \texttt{nodes} & \texttt{depth} & \texttt{lbtime} & \texttt{ubtime} & \texttt{time} &  \texttt{solved} \\
\hline
MVWB & 24.4 & 3.7 & 382.0 & 4.1 & 386.3 & 81/90 \\
MVSB & 9.9 & 2.7 & 125.1 & 1.4 & 467.6 & 80/90 \\
RBEB & 26.8 & 10.1 & 173.1 & 4.3 & 1121.1 & 72/90\\
\hline
\hline
\end{tabular}
\end{table}
\begin{figure}[h!]
  \centering
    \includegraphics[width=0.9\textwidth]{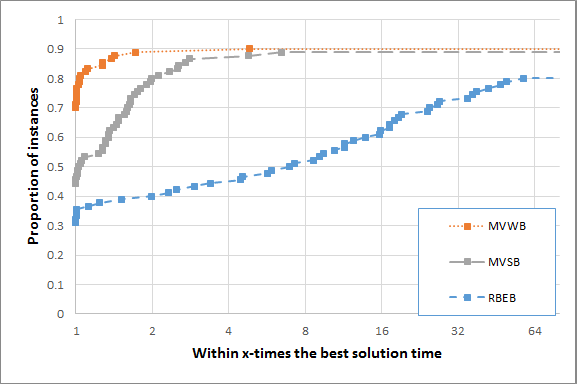}
      \caption{Performance profile of solve times for different branching rules on BoxQP.}
\label{fig:qptime}
\end{figure}
\begin{figure}[h!]
  \centering
    \includegraphics[width=0.9\textwidth]{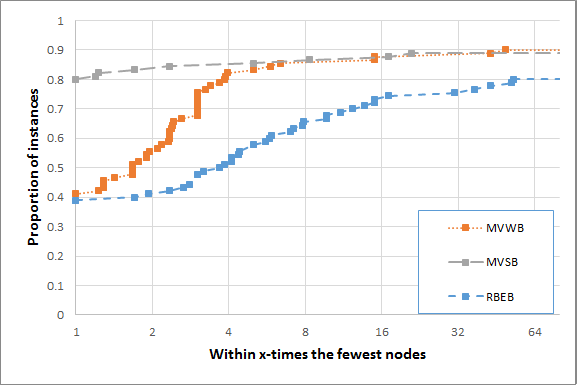}
      \caption{Performance profile of search tree nodes with respect to branching rules on BoxQP.}
\label{fig:qpnodes}
\end{figure}

\section{Conclusion}
\label{sec:conclusion}
We developed a spatial branch-and-cut method for generic Quadratically-Constrained Quadratic Programs with bounded complex variables. We derived valid inequalities from the convex hull description of nonconvex rank-one restricted sets to strengthen the SDP relaxations. We gave a new branching method based on an alternative characterization of a rank-one constraint. Experiments on Alternating Current Optimal Power Flow problems show the valid inequalities derived from the complex formulation are critical for improving the performance of the algorithm.  The proposed branching methods based on the rank-one constraint resulted in better performance compared to the benchmark reliability branching method.  Tests on box-constrained nonconvex Quadratic Programming instances suggest that the violation-based branching methods may also be effective for problems with real variables.

\appendix

\section{Sufficiency of Sparse Valid Inequalities}
\label{sec:suff}
Consider a Hermitian matrix $X$ with spectral decomposition $X=\sum_{k=1}^N \lambda_k d_kd_k^*$, where the eigenvalues are ordered so that $\lambda_k\geq \lambda_{k+1}$. Unlike in the real symmetric case, if $\lambda_k$ has multiplicity 1, then the eigenvector $d_k$ is only unique up to rotation by a complex phase $e^{j\theta_k}$ \cite[pp. 41]{terence2012topics}.  In polar coordinates (as used commonly in ACOPF) we have that the eigenvector is unique up to scaling of all phase angles by the same degree (preserving angle differences).  That is, if $d_{ki} = |d_{ki}|(\cos(\theta_{ki})+j\sin(\theta_{ki}))$, then we can add $\delta\in \mathbb{R}$ to all angles and replace $d_k$ in the eigenbasis.   In terms of rectangular coordinates (i.e. real and imaginary components), we can state that $d_k$ can be replaced in the eigenbasis with
\[\mbox{Re}(d_k) +\delta \vec{1} + j (\delta^I-\mbox{Im}(d_k)),\]
for any $\delta \in \mathbb{R}$ such that $(\mbox{Re}(d_{ki})+\delta)^2\leq |d_{ki}|^2 \forall i$ and $\delta^I \in \mathbb{R}^N$ with entries $\delta^I_i = \sqrt{|d_{ki}|^2-(\delta+\mbox{Re}(d_{ki}))^2}$.

Sparse positive semidefinite decomposition of a Hermitian matrix $X\in \mathbb{H}^{N\times N}$ yields a set of index sets $C$, where $X_c\succeq 0  \ \forall c\in C \iff X\succeq 0$.  Note that $\bigcap_{\forall c \in C} c = \{1,...,N\}$. A property of sparse decomposition is that $C$ can be represented with an acyclic graph where each node is an element $c$ of $C$ and an edge between two nodes indicates that at least one index is shared between the corresponding index sets; this is known as a clique tree  \cite{grone1984positive,fukuda2001exploiting}.  The clique tree of a chordal graph can be constructed in time and space linear with respect to the number of edges  \cite{blair1993introduction}.  A matrix can be completed for a certain property if there exist values for the entries not specified by the clique tree such that the fully specified matrix can attain the property.

\begin{proposition}
$X_c \succeq 0, \mbox{rank}(X_c)\leq 1 \ \forall c \in C$ iff $X$ can be completed so that $X=xx^*$ for some $x\in \mathbb{C^N}$.
\end{proposition}

\begin{proof}
If $X=xx^*$, then $\mbox{rank}(X)\leq 1, X\succeq 0$ and so one direction is obvious.  Now consider the other direction: suppose that $X_c \succeq 0, \mbox{rank}(X_c)\leq 1 \ \forall c \in C$.  We will use a constructive proof, i.e. we shall construct an $x$ so that $(xx^*)_c = X_c \forall c \in C$.

Let us consider the clique tree corresponding to the chordal graph formed by $C$.  Label a terminal node $c_1$.  Since $X_{c_1}$ has rank one and is positive semidefinite, we have that $X_{c_1} = \lambda_1^{c_1} (d_1^{c_1})^*d_1^{c_1}$, and so we can set $x_{c_1}= \sqrt{\lambda_1^{c_1}}d_1^{c_1}$ for any normalized principal eigenvector $d_1^{c_1}$. Now denote a neighbouring node $c_2$ and consider its corresponding index set with some normalized principal eigenvector, $d_2^{c_2}$.  By clique tree property, $X_{c_1}$ and $X_{c_2}$ share at least one entry, say $X_{mm}$, so $|(d_{1}^{c_1})_m|=|(d_{2}^{c_2})_m|$.  Since eigenvectors of the eigenbasis are only unique up to rotation by complex phase, $d_{2}$ can be rotated to form $\hat d_{2}$ so that $(d_{1}^{c_1})_m=(\hat d_{2}^{c_2})_m$, i.e. rotating the eigenvector so that one entry attains a specific angle.  Then we can set $x_{c_2} = \sqrt{\lambda_1^{c_2}}\hat d_{2}$, where $x_{mm}$, the shared entry of $x_{c_1},x_{c_2}$ retains the same value. The remaining elements of $x$ can be found by proceeding through neighbours in the same manner, with the acyclic property ensuring that each element of $x$ is set once.
 \end{proof}

This relies on a generalization of the fact that in ACOPF and in load flow the bus angles of any solution can be scaled up or down by constants. From the proposition it immediately follows that in the alternative rank condition only the $2\times 2$ principal minors related to the submatrices $X_c$ need to be considered, and so valid inequalities~(\ref{eq:vi1}) and (\ref{eq:vi2}) can be applied in a sparse fashion.

\section{ACOPF Formulation}
\label{sec:acopf}
ACOPF can be written in the form of CQCQP  \cite{lavaei_zero_2012}; we state a lifted formulation of the ACOPF problem as follows:

\begin{subequations}
\begin{alignat}{2}
\min  \ \ & c_2'[P+D_P]^2 + c_1'(P+D_P)+c_0 \\
 \text{s.t.} \ \  &P = \mbox{diag}(GW-BT) \label{eq:LACOPF1}\\
   &Q = \mbox{diag}(-BW-GT) \label{eq:LACOPF2}\\
   & P^{\min}\leq  P+D_P\leq P^{\max} \label{eq:LACOPF3a}\\
   & Q^{\min}\leq  Q+D_Q\leq P^{\max} \label{eq:LACOPF3b}\\
   &[V^{\min}]^2\leq  \mbox{diag}(W)\leq [V^{\max}]^2 \label{eq:LACOPF4}\\
 (\bold{LACOPF}) \ \ \ \ \  &\tan(\theta^{\min})W_{ij} \leq T_{ij} \leq  \tan(\theta^{\max})W_{ij} \label{eq:LACOPF5}\\
   &P_f = \mbox{diag}(C_f(G_fW-B_fT)) \label{eq:LACOPF6}\\
   &Q_f = \mbox{diag}(C_f(-B_fW-G_fT)) \label{eq:LACOPF7}\\
   &P_t = \mbox{diag}(C_t(G_tW-B_tT)) \label{eq:LACOPF8}\\
   &Q_t = \mbox{diag}(C_t(-B_tW-G_tT)) \label{eq:LACOPF9}\\
   &[P_f]^2 + [Q_f]^2\leq [S^{\max}]^2 \label{eq:LACOPF10a}\\
   &[P_t]^2 + [Q_t]^2\leq [S^{\max}]^2 \label{eq:LACOPF10b}\\
   &W+\I T\succeq 0 \label{eq:LACOPF11}\\
&\mbox{rank}(W+\I T)=1 \label{eq:LACOPF12}
\end{alignat}
\end{subequations}

Let $n$ be the number of nodes in the graph of the problem, with nodes representing either buses or transformers, and let $k$ be the number of edges (aka branches).
\subsubsection*{Variables and Data}
The decision variables used in LACOPF are: nodal powers $P+jQ \in \mathbb{C}^n$; a Hermitian decision matrix representing the outer product of nodal voltages, $W+\I T \in \mathbb{H}^{n\times n}$; and power to and from buses (respectively) across branches, $P_f+jQ_f,P_t+jQ_t \in \mathbb{C}^k$.  All other parameters are fixed data: convex costs $c_0 \in \mathbb{R}, c_1 \in \mathbb{R}^N,c_2 \in \mathbb{R}_+^N$; load, $D_P+jD_Q \in \mathbb{C}^n$; admittance matrices, $Y \in \mathbb{C}^{n\times n}, \ Y_f,Y_t \in \mathbb{C}^{k\times n}$; voltage magnitude limits, $V^{\min},V^{\max} \in \mathbb{R}^n$; phase angle limits $\theta^{\min},\theta^{\max} \in \mathbb{R}^k$; generator limits, $P^{\min}+jQ^{\min},P^{\max}+jQ^{\max}\in \mathbb{C}^{n}$; and line limits, $S^{\max} \in \mathbb{R}_{++}^n$.   $Y:=G+jB$ is the bus admittance matrix, and $Y_f:=G_f+jB_f,Y_t:=G_t+jB_t$ are branch admittances corresponding to `from' and `to' nodes, respectively.  Admittance is composed of conductance $G$ and susceptance $B$.  For a branch $r$ from $m$ to $n$, the $(r,m)$ entry of $C_f \in \mathbb{R}^{k\times n}$ and the $(r,n)$ entry of $C_t \in \mathbb{R}^{k\times n}$  are 1; all unconnected entries in those matrices are 0  \cite{zimmerman2011matpower}.

\subsubsection*{Objective and Constraints}
The objective is to minimize the cost of real power generation, where $P+P^D$ is the net generation of real nodal power.  Constraints~(\ref{eq:LACOPF1}) and (\ref{eq:LACOPF2}) are the power flow equations, relating nodal power to nodal voltage. Constraints~(\ref{eq:LACOPF3a}) and (\ref{eq:LACOPF3b}) model demand and generation limits.  Constraint~(\ref{eq:LACOPF4}) bounds voltage magnitudes.   Constraint~(\ref{eq:LACOPF5}) bounds phase angle differences between buses.  Constraints~(\ref{eq:LACOPF6}) to (\ref{eq:LACOPF9}) are the branch power flow equations. Constraints~(\ref{eq:LACOPF10a}) and (\ref{eq:LACOPF10b}) constrain apparent power across lines.  We will also consider other types of line limits, but omit them here for brevity. Constraints~(\ref{eq:LACOPF11}) and (\ref{eq:LACOPF12}) ensure that $W+\I T$ can represent the outer product of nodal voltages  \cite{lavaei_zero_2012}.

\subsubsection*{SDP Relaxation}
In LACOPF only the rank constraint~(\ref{eq:LACOPF12}) is nonconvex, and dropping it gives a convex relaxation (RACOPF).  This primal relaxation approach was first applied to ACOPF by Bai et al. \cite{bai_semidefinite_2008} and the conic dual formulation was first considered by Lavaei and Low \cite{lavaei_zero_2012}.  Note that in ACOPF the bounds $L_{ij},U_{ij}, i\neq j$ are specified as angle bounds in polar coordinates, i.e. $L_{ij}=\tan(\theta^{\min}_{ij}),U_{ij}=\tan(\theta^{\max}_{ij})$.

\vskip 5mm
\noindent
\textbf{Acknowledgements.}
The authors would like to thank Dr. Richard P. O'Neill of the Federal Energy Regulatory Commission for the initial impetus to study conic relaxations of ACOPF, and for helpful comments in early drafts of the paper. This research has been supported, in part, by Federal Energy Regulatory Commission and by grant FA9550-10-1-0168 from the Office of Assistant Secretary Defense for Research and Engineering. Chen Chen was supported, in part, by a NSERC PGS-D fellowship.

\bibliographystyle{plain}
\bibliography{sbc}

\end{document}